\documentclass[reqno, 11pt]{amsart}

\usepackage{mathrsfs}
\usepackage{amsthm}
\usepackage{amsmath,amssymb}
\usepackage[all,poly,knot]{xy}
\usepackage{amsfonts}

\newtheorem{thm}{Theorem}[section]
\newtheorem{cor}[thm]{{Corollary}}

\newtheorem{lem}[thm]{{Lemma}}
\newtheorem{prop}[thm]{Proposition}
\newtheorem{conj}[thm]{{Conjecture}}

\theoremstyle{remark}
\newtheorem*{rmk}{Remark}

\numberwithin{equation}{section}

  \def\Z{\mathbb Z}
  \def\Q{\mathbb Q}
  \def\A{\mathbb A}
  \def\F{\mathbb F}

  \def\Tr{\mathrm{Tr}}

  \newcommand{\NP}{\mathrm{NP}}
  \newcommand{\HP}{\mathrm{HP}}
  \newcommand{\Spec}{\mathrm{Spec}}
  \newcommand{\len}{\mathrm{Len}}
  \def\p{\mathfrak p}
  \def\P{\mathfrak P}
 \def\co{\mathcal{O}}

  \usepackage{changepage}

\begin{document}

\title{On a conjecture of Wan about limiting Newton polygons}

\author{Yi Ouyang}
\author{Jinbang Yang}

\address{Wu Wen-Tsun Key Laboratory of Mathematics, School of Mathematical Sciences, University of Science and Technology of China, Hefei, Anhui 230026, P. R. China}
\email{yiouyang@ustc.edu.cn, yjb@mail.ustc.edu.cn}
\thanks{Corresponding author: J. Yang. Email: yjb@mail.ustc.edu.cn}

\date{}
\maketitle
\begin{abstract} We show that for a monic polynomial $f(x)$ over a number field $K$ containing a global permutation polynomial of degree $>1$ as its composition factor, the Newton Polygon of $f\mod\p$ does not converge for $\p$ passing through all finite places of $K$.
In the rational number field case, our result is the ``only if" part of a conjecture of Wan about limiting Newton polygons.
\end{abstract}
\section{Introduction and main results}

Let $K$ be a number field and $f(x)$ be a monic polynomial in $K[x]$ of degree $d\geq1$. For a finite place $\p$ of $K$, we let $\co_\p$ be the ring of $\p$-adic integers and $k_\p$ be the residue field. Then $k_\p$  is a finite field of $q=q_\p=p^h$ elements for some rational prime $p=p_\p$ and some positive integer $h=h_\p$. Denote by $k_\p^m$ the unique field extension of $k_\p$ of degree $m$. Denote by $\Sigma_K$ the set of finite places of $K$ and $\Sigma_K(f)$ the set of places of $\p\in \Sigma_K$ such that $f(x)\in \co_\p[x]$ and $(d,p)=1$. Note that $\Sigma_K-\Sigma_K(f)$ is a finite set.

Let $\p$ be a place in $\Sigma_K(f)$. By modulo $\p$, we get the reduction $\overline f$ a polynomial over $k_\p$. For a nontrivial character $\chi: \F_p\rightarrow \mu_p$, the $L$-function
\begin{equation}
L(\overline f,\chi,t)=L(\overline f/k_\p,\chi,t)=\exp\left(\sum_{m=1}^\infty S_m(\overline f,\chi)\frac{t^m}{m}\right),
\end{equation}
where $S_m(\overline f,\chi)$ is the exponential sum
\begin{equation} S_m(\overline f,\chi)=S_m(\overline f/k_\p,\chi)=\sum_{x\in k^m_\p} \chi(\mathrm{Tr}_{k^m_\p/\F_p}(\overline f(x))),
\end{equation}
is a polynomial of $t$ of degree $d-1$ over $\Q_p(\zeta_{p})$ by well-known theorems of Dwork-Bombieri-Grothendieck and Adolphson-Sperber \cite{AS87}. The  $q$-adic Newton polygon $\NP_\p(f)$ of this $L$-function does not depend on the choice of the nontrivial character $\chi$.

 Let $\HP(f)$ be a convex polygon with break points
\[\left\{(0,0),\Bigl(1,\frac1d\Bigr),\Bigl(2,\frac1d+\frac2d\Bigr),\cdots,\Bigl(d-1,\frac1d+\frac2d+\cdots+\frac{d-1}d\Bigr)\right\},\]
which only depends on the degree of $f$. Adolphson and Sperber~\cite{AS89} proved that $\NP_\p(f)$ lies above $\HP(f)$ and that $\NP_\p(f)=\HP(f)$ if $p\equiv 1\mod d$.  Obviously, there are infinitely many $\p\in\Sigma_K(f)$ such that $p\equiv 1\mod d$, thus if $\lim\limits_{\p\in\Sigma_K} \NP_\p(f)$  exists, then $\lim\limits_{\p\in\Sigma_K} \NP_\p(f)=\HP(f)$.

Recall that a global permutation polynomial (GPP) over $K$ is a polynomial $P(x)\in K[x]$ such that $x\mapsto \overline{P}(x)$, where $\overline{P}$ is the reduction of $P$ modulo $\p$,  is a permutation on $k_\p$ for infinitely many places $\p\in \Sigma_K$.

In 1999, D. Wan proposed a conjecture, whose complete version in \cite[Chapter 5]{Yang03}  and \cite[Conjecture 6.1]{BFZ08} is as follows:
\begin{conj}[Wan]\label{conj:Wan} Let $f$ be a non-constant monic polynomial in $\Q[x]$.  Then $f$ contains a GPP over $\Q$ of degree $>1$ as its composition factor  if and only if
$\lim\limits_{\p\in\Sigma_\Q} \NP_\p(f)$ does not exist.
\end{conj}

There are little progress on the ``if" part, which is much more difficult than the ``only if" part.
So far, we can only check the ``if" part holds for those $f$ of low degrees or of few terms.  
In this note, we give a proof of the ``only if'' part of Wan's conjecture. Moreover, we get the following main result.
\begin{thm}\label{Mainthm}
	Let $f$ be a non-constant monic polynomial in $K[x]$.   If $f$ contains a GPP over $K$ of degree $>1$ as its composition factor, then
	$\lim\limits_{\p\in\Sigma_K} \NP_\p(f)$ does not exist.
\end{thm}

\begin{rmk}
	 If we replace $\Q$ in Conjecture~\ref{conj:Wan} by any number field $K$, then the ``if" part does not hold in general. We give an example here.  Let $\ell$ be a prime number greater than $3$. Assume $K=\Q(\zeta_\ell)$ and $f(x)=$ the Dickson polynomial $D_\ell(x,1)$. By Lemma~\ref{lem:AdmTri}, $f$ is not a permutation polynomial for all $k_\p$ with $\p\nmid 3\ell\omega$. Thus $f$ is not a GPP over $K$. By Lemma~\ref{lem:Dick&PP}, one can easily check that $f$ is a GPP over $\Q$.  Theorem~\ref{Mainthm} implies that $\lim\limits_{p\in\Sigma_\Q} \NP_p(f)$ does not exist. By Proposition~\ref{prop:ExtofBasField}, $\lim\limits_{\p\in\Sigma_K} \NP_\p(f)$ also does not exist.
\end{rmk}

\section{Zeta functions and $L$-functions of exponential sums}

In this section, we fix a rational prime $p$, a positive integer $h$ and let $q=p^h$.
Let $C$ be a curve over $\F_q$. The Zeta function of $C$
 \begin{equation}
 Z(C,t)=\exp\left(\sum_{m=1}^\infty S_m(C)\frac{t^m}{m}\right)
 \end{equation}
 is a rational function over $\Q$, where
 \[S_m(C)=\# C(F_{q^m})\]
 is the number of $\F_{q^m}$-rational points of $C$. If $C$ is smooth and proper, by Weil~\cite{Weil49}, $Z(C,t)$ is of the form $\frac{P_1(C)}{(1-t)(1-qt)}$, where $P_1(C)$ is a polynomial of $t$ of degree $2g(C)$ over $\Z$, where $g(C)$ is the genus of $C$. Denote the $q$-adic Newton polygon of $P_1(C)$ by $\NP_q^1(C)$.

Let $g$ be a polynomial in $\F_q[x]$ of degree $d$ with $(d,p)=1$.  The fraction field of the integral domain $\F_q[x,y]/(y^p-y-g)$, denoted by $L_g$, is a Galois extension of $\F_q(x)$, which is the function field of $\mathbb P_{\F_q}^1$. So $C(g)$, the normalization of $\mathbb P_{\F_q}^1$ in $L_g$, is a Galois cover of $\mathbb P_{\F_q}^1$ with Galois group isomorphic to $\F_p$. Denote this cover by $\pi$. One can check that $\pi^{-1}(\infty)$ is a one-point-set. The complement $U$ of $\pi^{-1}(\infty)$ in $C(g)$ is  $\Spec (\F_q[x,y]/(y^p-y-g))$.

In the following we identify $(x_0,y_0)$ with the $\overline\F_q$-point $(x-x_0,y-y_0)$ of $U_{\overline\F_q}$ and identify $x_0$ with the $\overline\F_q$-point $(x-x_0)$ of $\A_{\overline\F_q}^1$ for any $x_0,y_0\in \overline\F_q$ such that $y_0^p-y_0=g(x_0)$.
Obviously, for any point $x_0$, there is some $y_0\in\overline{F}_q$ such that the set $\pi^{-1}(x_0)$ is of the form \[\{(x_0,y_0),(x_0,y_0+1),\cdots,(x_0,y_0+p-1)\}.\]

\begin{lem}\label{lem:fible}
	Assume that $x_0\in \F_{q^m}$. Then the number of $\F_{q^m}$-points in $\pi^{-1}(x_0)$ is
	\[\sum_{\chi} \chi(\Tr_{\F_{q^m}/\F_p}(g(x_0))),\]
	 where $\chi$ runs through all additive characters from $\F_p$ to $\mu_p$.
\end{lem}

\begin{proof}
For any $a\in\F_p$, one can easily check that
	\[\sum_{\chi} \chi(a)=\begin{cases}
	 0, & \text{if }  a\neq0;\\
	 p, & \text{if }  a=0.\\	
	\end{cases}\]
Let $(x_0,y_0)$ be a point in $\pi^{-1}(x_0)$. We only need to show  that $y_0\in \F_{q^m}$ if and only if $\Tr_{\F_{q^m}/\F_p}(g(x))=0$. This follows from the following exact sequence
\[0\rightarrow \F_p\overset{\text{inc}}{\longrightarrow} \F_{q^m} \rightarrow \F_{q^m}\overset{\Tr}{\longrightarrow} \F_p\rightarrow 0,\]
where the middle map is given by $a\mapsto a^p-a$.
\end{proof}

\begin{prop}
	$P_1(C(g),t)=\prod\limits_{\chi\neq1} L(g,\chi,t)$.
\end{prop}
\begin{proof} By Lemma~\ref{lem:fible},
	\[S_m(U)=\sum_{x_0\in \F_{q^m}}\sum_{\chi} \chi(\Tr_{\F_{q^m}/\F_p}(g(x_0)))=q^m+\sum_{\chi\neq1} S_m(g,\chi).\]
	As $S_m(C(g))=1+S_m(U)$ and $P_1(C(g),t)=(1-t)(1-qt)Z(C(g),t)$, by definition of Zeta functions,
\[P_1(C(g),t) =(1-t)(1-qt)\times \exp\left[\sum_{m=1}^\infty\left(1+q^m+\sum_{\chi\neq1} S_m(g,\chi)\right)\frac{t^m}{m}\right].\]
	The Proposition follows from the definition of $L$-functions.
\end{proof}

For any polygon $P$, denote by $\len(P,\lambda)$ the length of the side of slope of $\lambda$. As the Newton polygon $\NP_\p(f)$ of $L(\overline f,\chi,t)$ does not depend on the choice of $\chi\neq1$, we have the following result.
\begin{cor}\label{cor:Zeta&Lfun} For any $\lambda$, $\len(\NP_q^1(C(\overline f)),\lambda)=(p-1)\len(\NP_\p(f),\lambda)$.	
\end{cor}

\begin{lem}\label{lem:RootLfun&ExpSum}
	Write $L(g,\chi,t)$ in the form $(1-\alpha_1t)(1-\alpha_2t)\cdots(1-\alpha_{d-1}t)$. For any $m\geq 1$, we have
	\[S_m(g,\chi)=-(\alpha_1^m+ \alpha_2^m+ \cdots + \alpha_{d_1-1}^m).\]	
\end{lem}
\begin{proof}By definition of $L(g,\chi,t)$,
	\[(1-\alpha_1t)(1-\alpha_2t)\cdots(1-\alpha_{d-1}t)=\exp\left(\sum_{m=1}^\infty S_m(g,\chi)\frac{t^m}{m}\right).\]
	Taking logarithm and expending both sides, we can get the formula by comparing the coefficients of $t^m$ on both sides.
\end{proof}

\begin{lem}\label{lem:ExtofCoeff}
	Write $L(g,\chi,t)$ in the form $(1-\alpha_1t)(1-\alpha_2t)\cdots(1-\alpha_{d-1}t)$. For any $n\geq 1$, we have 
	\[L(g/\F_{q^n},\chi,t)=(1-\alpha_1^nt)(1-\alpha_2^nt)\cdots(1-\alpha_{d-1}^nt).\]
	In particular, the $q$-adic Newton polygon of $L(g,\chi,t)$ is the same as the $q^n$-adic Newton polygon of $L(g/\F_{q^n},\chi,t)$.
\end{lem}

\begin{proof}
	Assume that $L(g/\F_{q^n},\chi,t)=(1-\beta_1t)(1-\beta_2t)\cdots(1-\beta_{d-1}t)$. It is clear that 
	\[S_m(g/\F_{q^n},\chi)=S_{mn}(g,\chi)\]
	holds for any $m\geq0$.	By Lemma~\ref{lem:RootLfun&ExpSum}, we have 
	\[\beta_1^m+\beta_2^m+\cdots+\beta_{d-1}^m=\alpha_1^{mn}+\alpha_2^{mn}+\cdots+\alpha_{d-1}^{mn}.\]
	Hence we have 
	\[\sum_{m=0}^{\infty}(\beta_1^m+\beta_2^m+\cdots+\beta_{d-1}^m)t^m=\sum_{m=0}^{\infty}(\alpha_1^{mn}+\alpha_2^{mn}+\cdots+\alpha_{d-1}^{mn})t^m\]
	That is 
	\[\frac{1}{1-\beta_{1}t}+\frac{1}{1-\beta_{2}t}+\cdots+\frac{1}{1-\beta_{d-1}t}
	 =\frac{1}{1-\alpha_{1}^nt}+\frac{1}{1-\alpha_{2}^nt}+\cdots+\frac{1}{1-\alpha_{d-1}^nt}.\]
	 Comparing the poles on both sides, we are done.
\end{proof}

\begin{prop}\label{prop:ExtofBasField}
	Let $L/K$ be a finite extension of number fields and $\P$ a place of $L$ above $\p$ a place of $K$. Then
	\[\NP_\p(f)=\NP_\P(f).\] 
	In particular, $\lim\limits_{\p\in\Sigma_K} \NP_\p(f)$ exists if and only if $\lim\limits_{\P\in\Sigma_L} \NP_\P(f)$ exists
\end{prop}

\begin{proof}
	By definition, $\NP_\p(f)$ is the $q$-adic Newton polygon of $L(\overline{f}/k_\p,\chi,t)$ and $\NP_\P(f)$ is the $q^{[k_\P:k_\p]}$-adic Newton polygon of $L(\overline{f}/k_\P,\chi,t)$. By Lemma~\ref{lem:ExtofCoeff}, we have $\NP_\p(f)=\NP_\P(f)$. 
\end{proof}

\section{Divisibility relations of Zeta functions}

 We fix $p$, $h$ and $q=p^h$ as in the previous section. Let $X,Y$ be two smooth separated algebraic varieties over $\F_q$. Let $\pi:Y\rightarrow X$ be an $\F_q$-morphism and $\mathscr F$ be a sheaf over the \'etale site $\overline X_{et}$, where $\overline X=X_{\overline \F_q}$. The morphism $\pi_{\overline{\F}_q}$ induces a map $\pi^*: H^r_{c}(\overline X_{et},\mathscr F)\rightarrow H^r_{c}(\overline Y_{et},\pi^*\mathscr F)$. If $\pi$ is an \'etale morphism, then there is a natural isomorphism $ H^r_{c}(\overline Y_{et},\pi^*\mathscr F)\overset{\mathrm{can}}\longrightarrow H^r_{c}(\overline X_{et},\pi_*\pi^*\mathscr F)$.

\begin{lem}\label{lem:inj} Suppose that $\pi: Y\rightarrow X$ is a finite \'etale $\F_q$-morphism of degree $\delta$. Then there is a trace map $\mathrm{tr}: \pi_*\pi^*\mathscr F \rightarrow \mathscr F$ such that	
	\[\mathrm{tr}\circ\mathrm{can}\circ\pi^*=\delta.\]
In particular, if $\delta$ is invertible on $\mathscr F$, then $\pi^*$ is injective.
\end{lem}
\begin{proof} See pages 168-171 in \cite{Milne80}. \end{proof}

By the Grothendieck-Lefschetz trace formula (see~\cite{Grothendieck65}), the number $N_s$ of $\F_{q^s}$-rational points on $X$ is
\[N_s=\sum_{i=0}^{2d} (-1)^i \mathrm{Tr}((F^*)^s, H^i_c(\overline X_{et},\Q_\ell)),\]
where $F$ is the Frobenius endomorphism on $X/\F_q$.  So the Zeta function on $X$ is
\[Z(X,t)=\frac{P_1(X,t)P_3(X,t)\cdots P_{2d-1}(X,t)}{P_0(X,t)P_2(X,t)\cdots P_{2d}(X,t)}\]
with $P_i(X,t)=\det(1-tF^*\mid H^i_c(\overline X_{et},\Q_\ell))$.

\begin{thm}\label{prop:div}
	If there is some finite \'etale morphism $\pi:Y\rightarrow X$, then
	\[P_i(X,t)\mid P_i(Y,t).\]
\end{thm}

\begin{proof} As $F$ commutes with $\pi$, we have the following commutative diagram	\begin{equation*}
	\xymatrix{ H^i_c(\overline X_{et},\Q_\ell) \ar[r]^{\pi^*}\ar[d]^{F^*}  & H^i_c(\overline Y_{et},\Q_\ell)  \ar[d]^{F^*} \\
		H^i_c(\overline X_{et},\Q_\ell) \ar[r]^{\pi^*} & H^i_c(\overline Y_{et},\Q_\ell).\\}
\end{equation*}
	By Lemma~\ref{lem:inj}, $\pi^*$ is injective. The commutative diagram implies that  $H^i_c(\overline X_{et},\Q_\ell)$ can be viewed as an invariant subspace of $H^i_c(\overline Y_{et},\Q_\ell)$ under the action of $F^*$. So we have
	\[\det(1-tF^*\mid H^i_c(\overline X_{et},\Q_\ell))\mid \det(1-tF^*\mid H^i_c(\overline Y_{et},\Q_\ell)).\qedhere\]
\end{proof}

\begin{cor}\label{cor:DivZeta}
	Let $X,Y$ be two smooth complete curves over $\F_q$. If there is some finite $\F_q$-morphism $\pi:Y\rightarrow X$, then
	\[P_1(X,t)\mid P_1(Y,t).\] 	 	
\end{cor}
\begin{proof}
	
	By removing the compositions of Frobenius on $X$, we can assume that $\pi$ is unramified at the generic point. Let $U\subsetneq X$ be an nonempty open subvariety of $X$ such that the base change $\pi: Y_U\rightarrow U$ is a finite \'etale morphism.  Denote by $Z$ the complement of $U$ in $X$, of which the closed points are finite. By the definition of Zeta function, we have
	\[Z(X,t)=Z(U,t)\times\prod_{x}\frac1{1-t^{\deg x}}\]
	and
	\[Z(Y,t)= Z(Y_U,t)\times \prod_{y}\frac1{1-t^{\deg y}},\]
	where $x$ (resp. $y$) runs through all prime $\F_q$-rational $0$-cycles on $Z$ (resp. $Y_Z$) a la Monsky.
	As $X,Y$ are complete curves,
	\[	Z(X,t)=\frac{P_1(X,t)}{(1-t)(1-qt)},\quad Z(Y,t)=\frac{P_1(Y,t)}{(1-t)(1-qt)}.\]
	As $U,Y_U$ are not complete, we have
	\[	Z(U,t)=\frac{P_1(U,t)}{1-qt},\quad Z(Y_U,t)=\frac{P_1(Y_U,t)}{1-qt}.\]
	By Theorem~\ref{prop:div},  $P_1(U,t)\mid P_1(Y_U,t)$. So from the above formulas, we have
	\[P_1(X,t)\frac{\prod_{x} (1-t^{\deg x})}{1-t}\ \left| \ P_1(Y,t)\frac{\prod_{y} (1-t^{\deg y})}{1-t}\right. . \]
	Weil's conjecture tells us that the complex absolute values of all roots of $P_1(X,t)$ and $P_1(Y,t)$ are $q^{-\frac12}$. We are done.
\end{proof}

\section{Global permutation polynomials and Dickson polynomials}
Let $a$ be an element in a commutative ring $R$. For any $n\geq 1$, the Dickson polynomial of the first kind associated to $a$ of degree $n$, denote by $D_n(x,a)$, is the unique polynomial over $R$ such that
\begin{equation}D_n\Bigl(x+\frac ax,a\Bigr)=x^n+\frac{a^n}{x^n}. \end{equation}

One can easily check that
\begin{equation}D_n(x,0)=x^n \end{equation}
and
\begin{equation}\label{equ:DecDickPoly}D_{mn}(x,a)=D_m(D_n(x,a),a^n). \end{equation}

\begin{lem} \label{lem:Dick&PP}
	Let $a\in \F_q$ and $n$ be a positive integer.
	
	$1).$ If $a=0$, then $D_n(x,0)=x^n$ is a permutation polynomial of $\F_q$ if and only if  $(n,q-1)=1$.
	
	$2).$  If $a\neq0$, then $D_n(x,a)$ is a permutation polynomial of $\F_q$ if and only if $(n,q^2-1)=1$.
\end{lem}

\begin{proof} Due to \cite{Dickson1896}, see \cite[Theorem 7.16]{LiNi83} for quick reference.
\end{proof}

\begin{prop}[Fried-Turnwald] \label{prop:Turn}
	Let $f$ be a GPP over $K$. Then $f$ is a composition of linear polynomials $\alpha_ix+\beta_i\in K[x]$ and the Dickson polynomials $D_{n_j}(x,a_j)$, where $a_j\in K$ and $n_j$ are positive integers.	
\end{prop}
\begin{proof}See \cite[Theorem 2]{Fried70} or \cite[Theorem 2]{Turnwald95}.\end{proof}

\section{Proof of  main result}

We call an element $(\p,a,n)$ in $\Sigma_K\times K\times \Z_{>1}$ an \emph{admissible triple} if $a\in \co_\p$, $\p\nmid 3n\omega$ and the Dickson polynomial $D_n(x,\overline a)$ is a permutation polynomial on $k_\p$, where  $\omega$ is the number of the roots of unity in $K$.

\begin{lem}\label{lem:AdmTri}
	 If $(\p,a,n)$ is an  admissible triple, then $(n,\omega)=1$. In particular, $2\nmid n$. Moreover if $a\neq 0$, then $3\nmid n$.
\end{lem}

\begin{proof}
	As $D_n(x,\overline a)$ is a permutation polynomial on $k_\p$, by Lemma~\ref{lem:Dick&PP}, $(n,q-1)=1$. As $\p\nmid \omega$, the reduction induces an inclusion $\mu_K\subset \mu_{k_\p}$, and hence $\omega\mid q-1$. So we have $(n,\omega)=1$.
	
	If $a\neq0$, by Lemma~\ref{lem:Dick&PP}, $(n,q^2-1)=1$. As $3\mid q^2-1$, so we have $3\nmid n$.
\end{proof}


\begin{prop}\label{prop:main}
Suppose that $f$ contains $D_n(x,a)$ as a composition factor. Then for $\p\in\Sigma_K(f)$ such that $(\p,a,n)$ is an admissible triple, there exists $v_0\in \Q$ such that  $\len(\NP_\p(f),v_0)\geq 2$ and hence the gap between $NP_\p(f)$ and $\HP(f)$ is at least $\frac1{2d}$.
\end{prop}
\begin{proof}
	 Write $f$ in the form $f_1\circ D_n(x,a)\circ f_3$. By~(\ref{equ:DecDickPoly}) and Lemma~\ref{lem:AdmTri}, we can assume that $n$ is an odd prime number.  For any positive integer $m$, denote by $k_\p^m$ the unique extension of $k_\p$ of degree $m$. Set $e=1$ if $\overline{a}=0$ and otherwise $e=2$. By Lemma~\ref{lem:Dick&PP}, we have $(q^e-1,n)=1$. As $n$ is an odd prime number,  $(q^{(n-1)s+1})^e\equiv q^e\not\equiv1\mod n$ and so $((q^{(n-1)s+1})^e-1,n)=1$. Using Lemma~\ref{lem:Dick&PP} again,  $D_n(x,\overline a)$ is permutation polynomial of $k_\p^{m}$,  where $m=(n-1)s+1$ and $s$ is a non-negative integer.  For these $m$ and any nontrivial character $\chi:\F_p\rightarrow \mu_p$, we have that
	 \begin{equation}\label{equ:51}
	 S_m(\overline f_1,\chi)=S_m(\overline f_1\circ D_n(x,\overline a),\chi).
	 \end{equation}
 Assume that
\[L(\overline f_1,\chi,t)=(1-\alpha_1t)(1-\alpha_2t)\cdots(1-\alpha_{d_1-1}t)\]
and
\[L(\overline f_1\circ D_n(x,\overline a),\chi,t)=(1-\beta_1t)(1-\beta_2t)\cdots(1-\beta_{nd_1-1}t),\]
where $d_1$ is the degree of $f_1$. Lemma~\ref{lem:RootLfun&ExpSum} implies that
 \[S_m(\overline f_1,\chi)=-(\alpha_1^m+ \alpha_2^m+ \cdots + \alpha_{d_1-1}^m)\]
 and
\[S_m(\overline f_1\circ D_n(x,\overline a),\chi)= -(\beta_1^m+\beta_2^m + \cdots + \beta_{nd_1-1}^m).\]	

By \eqref{equ:51}, we have an equality of power series
\[\sum_{m=(n-1)s+1}(\alpha_1^m+ \alpha_2^m+ \cdots + \alpha_{d_1-1}^m)t^m=\sum_{m=(n-1)s+1}(\beta_1^m+\beta_2^m + \cdots + \beta_{nd_1-1}^m)t^m.\]
Hence
\[\sum_{i=1}^{d_1-1}\frac{\alpha_it}{1-(\alpha_it)^{n-1}}=\sum_{i=1}^{nd_1-1}\frac{\beta_it}{1-(\beta_it)^{n-1}}.\]
Comparing the poles on both sides, there exist $1\leq i<j\leq nd_1-1$ such that
\[\beta_i^{n-1}=\beta_j^{n-1}.\]
Denote by $v_0$ the $q$-adic valuation  of $\beta_i$ (and of $\beta_j$). Then
\[\len(\NP_\p(f_1\circ D_n(x,a)),v_0)\geq 2.\]
Denote $C'=C(\overline f_1\circ D_n(x,\overline a))$, by Corollary~\ref{cor:Zeta&Lfun},
\[\len(\NP_q^1(C'),v_0)\geq 2(p-1).\]

Denote $C=C(f)$, one can check that
\[k_\p(C')=k_\p(x,y')\text{ and } k_\p(C)=k_\p(x,y),\] where $(y')^p-y'=\overline f_1\circ D_n(x,\overline a)$ and $y^p-y=f(x)$.
The embedding
\[k_\p(x,y')\rightarrow k_\p(x,y)\]
sending  $x$ to $\overline{f_3}$ and $y'$ to $y$ induces a non-constant morphism
\[\pi:C\rightarrow C'\]
of complete smooth curves. By Proposition~\ref{cor:DivZeta},
\[\len(\NP_q^1(C),v_0)\geq\len(\NP_q^1(C'),v_0)\geq 2(p-1).\]
 Using Corollary~\ref{cor:Zeta&Lfun} again, we have
 \[\len(\NP_\p(f),v_0)\geq 2.\]
 \begin{center}
 	\setlength{\unitlength}{1.5mm}
 	 \begin{picture}(60,25)
 	 \linethickness{1pt}
 	 \put(0,0){\vector(1,0){60}}
 	 \put(0,0){\vector(0,1){25}}
 	 \thicklines
 	 \put(15,2){\line(4,1){12}}	
 	 \put(27,5){\line(4,3){12}}
 	 \put(15,8){\line(2,1){24}}	
 	 \multiput(15,0)(0,0.5){16}{\line(0,1){0.3}}
 	 \multiput(27,0)(0,0.5){28}{\line(0,1){0.3}}
 	 \multiput(39,0)(0,0.5){40}{\line(0,1){0.3}}
 	 \multiput(15,2)(0.4,0.2){60}{\line(0,1){0.3}}
 	 \put(34,4){HP}
 	 \put(24,18){NP}
 	 \put(27,14){\circle*{0.7}}
 	 \put(27.5,12.5){$N_i$}
 	 \put(27,8){\circle*{0.7}}
 	 \put(27.5,6.5){P}
 	 \put(27,5){\circle*{0.7}}
 	 \put(27.5,3.5){$H_i$}
 	 \put(15,8){\circle*{0.7}}
 	 \put(15.5,6.5){$N_{i-1}$}
 	 \put(15,2){\circle*{0.7}}
 	 \put(15.5,0.5){$H_{i-1}$}
 	 \put(39,20){\circle*{0.7}}
 	 \put(39.5,18.5){$N_{i+1}$}
 	 \put(39,14){\circle*{0.7}}
 	 \put(39.5,12.5){$H_{i+1}$}
 	 \end{picture}
 \end{center}
As in the above diagram, we assume that $N_{i-1}N_i$ and $N_iN_{i+1}$ are of the same slope. The slopes of $H_{i-1}H_i$ and $H_iH_{i+1}$ are $\frac id$ and $\frac{i+1}{d}$, respectively. As the HP is below the NP, we know that $N_{i\pm1}$ is above $H_{i\pm1}$. Hence the middle point $N_i$ of $N_{i-1}N_{i+1}$ is above $P$ that of $H_{i-1}H_{i+1}$. So we have
\[|N_iH_i|\geq |PH_i|\geq \frac1{2d}.\qedhere \]
\end{proof}

\begin{proof}[Proof of Main Result.]  Write $f$ in the form $f_1\circ f_2\circ f_3$, where $f_2$ is a GPP over $K$ of degree $>1$. As every composition factor of a GPP is still a GPP,  by Proposition~\ref{prop:Turn}, we can assume that $f_2=D_n(x,a)$ is a GPP over $K$, where $a\in K$ and $n\in \Z_{>1}$.
	
By  definition, there are infinitely many $\p\in\Sigma_K$ such that $(\p,n,a)$ is an admissible triple. For those $\p$,  by Proposition~\ref{prop:main}, the gap between $NP_\p(f)$ and $\HP(f)$ is at least $\frac1{2d}$. However, for places $\p$ such that  $p_\p\equiv 1\mod d$, we know $NP_\p(f)=\HP(f)$. So the limit does not exist.
\end{proof}

\subsection*{Acknowledgement} Research is partially supported by National Key Basic Research Program of China (Grant No. 2013CB834202) and National Natural Science Foundation of China (Grant No. 11171317 and 11571328).

 \end{document}